\documentclass[12pt]{article}
\usepackage{amsmath,latexsym,amssymb}
\usepackage{amsthm}
\usepackage{enumerate}

\newcommand\NN{\mathbb{N}}
\newcommand\RR{\mathbb{R}}
\newcommand\ZZ{\mathbb{Z}}

\newcommand\BB{\mathcal{B}}

\newcommand\eps{{\varepsilon}}
\DeclareMathOperator\dist{dist}
 
\DeclareMathOperator\diam{diam}
\DeclareMathOperator\defin{def}

\newtheorem{theorem}{Theorem}%[section]
\newtheorem{definition}[theorem]{Definition}

\newtheorem{lemma}[theorem]{Lemma}

\newtheorem{proposition}[theorem]{Proposition}

\newcommand{\address}{Address: Department of Mathematics, University of North Texas, 1155 Union Circle \#311430, Denton, TX 76203-5017, USA; E-mail: allaart@unt.edu}

\title{Correction and strengthening of ``How large are the level sets of the Takagi function?"}
\author{Pieter C. Allaart \footnote{\address}}

\begin{document}

\maketitle

\begin{abstract}

The purpose of this note is to correct an error in an earlier paper by the author about the level sets of the Takagi function ({\em Monatsh. Math.} {\bf 167} (2012), 311-331), and to prove a stronger form of one of the main results of that paper.


\end{abstract}

Takagi's continuous nowhere differentiable function is defined by
\begin{equation}
T(x)=\sum_{n=0}^\infty \frac{1}{2^n}\phi(2^n x),
\label{eq:Takagi-def}
\end{equation}
where $\phi(x)=\dist(x,\ZZ)$, the distance from $x$ to the nearest integer. In \cite{Allaart}, a number of results were proved about the level sets
\begin{equation}
L(y):=\{x\in[0,1]: T(x)=y\}, \qquad y\in\RR.
\end{equation}
The purpose of this note is two-fold: First, to point out an error in \cite{Allaart}, which leads to the loss of some statements pertaining to the set-theoretic complexity of certain sets. A few other results require new proofs, which are given here; Second, to strengthen the last main result of \cite{Allaart} concerning the number of local level sets contained in a `typical' level set.

\section{Preliminaries}

We first recall some definitions and lemmas from \cite{Allaart}. To avoid confusion with cross-referencing, definitions, lemmas, etc. in this note are labeled by a single number. Thus, for instance, ``Proposition \ref{lem:bijection}" refers to Proposition \ref{lem:bijection} of the present note, whereas ``Theorem 4.6" refers to a theorem in \cite{Allaart}. We shall let $\NN$ denote the set of positive integers, and $\ZZ_+$ the set of nonnegative integers.

We write the binary expansion of $x\in[0,1)$ as
\begin{equation*}
x=\sum_{n=1}^\infty \frac{\eps_n}{2^n}=0.\eps_1\eps_2\dots\eps_n\dots, \qquad\eps_n\in\{0,1\},
\end{equation*}
choosing the representation ending in all zeros if $x$ is dyadic rational. For $k\in\ZZ_+$ and $x$ written as above, let
\begin{equation*}
D_k(x):=\sum_{j=1}^k(-1)^{\eps_j}.
\end{equation*}
%A crucial property of the Takagi function is the following fact; see \cite[Lemma 2.1]{Allaart}.
%\begin{lemma} \label{lem:equivalence}
%If $|D_n(x)|=|D_n(x')|$ for every $n$, then $T(x)=T(x')$.
%\end{lemma}
Define an equivalence relation on $[0,1)$ by
\begin{equation*}
x\sim x'\quad \stackrel{\defin}{\Longleftrightarrow}\quad |D_j(x)|=|D_j(x')|\ \mbox{for each $j\in\NN$}.
\end{equation*}

\begin{definition}[Lagarias and Maddock \cite{LagMad1}]
{\rm
The {\em local level set} containing $x\in[0,1)$ is defined by 
\begin{equation*}
L_x^{loc}:=\{x'\in [0,1): x'\sim x\}.
\end{equation*}
}
\end{definition}

It was shown in \cite{LagMad1} (see also \cite[Lemma 2.1]{Allaart}) that
\begin{equation}
x\sim x' \quad \Longrightarrow \quad T(x)=T(x').
\label{eq:equivalence}
\end{equation} 
Thus, each local level set is a subset of a level set. The local level set $L_x^{loc}$ is finite if $D_j(x)=0$ for only finitely many $j$, and is a Cantor set otherwise. (Lagarias and Maddock defined local level sets slightly differently, essentially treating them as subsets of the Cantor space $\{0,1\}^\NN$. Our definition here is simpler, and the difference in definitions does not affect which local level sets are finite, or how many local level sets are contained in a level set. See \cite[Remark 3.10]{Allaart} for more detail.)

\begin{definition}
{\rm
A dyadic rational of the form $x=0.\eps_1\eps_2\dots\eps_{2m}$ is called {\em balanced} of {\em order} $m$ if $D_{2m}(x)=0$. If there are exactly $n$ indices $1\leq j\leq 2m$ such that $D_j(x)=0$, we call $n$ the {\em generation} of $x$. By convention, we consider $x=0$ to be a balanced dyadic rational of generation $0$. For $n\in\ZZ_+$, the set of balanced dyadic rationals of generation $n$ is denoted by $\BB_n$, and we put $\BB:=\bigcup_{n=0}^\infty \BB_n$.
}
\end{definition}

Let
\begin{equation*}
\mathcal{G}_T:=\{(x,T(x)): 0\leq x\leq 1\}
\end{equation*}
denote the graph of $T$ over the unit interval $[0,1]$. The next result is Lemma 2.3 of \cite{Allaart}.

\begin{lemma} \label{lem:similar-copies}
Let $m\in\NN$, and let $x_0=k/2^{2m}$ be a balanced dyadic rational of order $m$. Then for $x\in[k/2^{2m},(k+1)/2^{2m}]$, we have
\begin{equation*}
T(x)=T(x_0)+\frac{1}{2^{2m}}T\left(2^{2m}(x-x_0)\right).
\end{equation*}
In other words, the part of $\mathcal{G}_T$ above the interval $[k/2^{2m},(k+1)/2^{2m}]$ is a similar copy of $\mathcal{G}_T$ itself, reduced by a factor $2^{2m}$ and shifted up by $T(x_0)$.
\end{lemma}

\begin{definition} \label{def:humps}
{\rm
For a balanced dyadic rational $x_0=k/2^{2m}$ as in Lemma \ref{lem:similar-copies}, define 
\begin{align*}
I(x_0)&:=[k/2^{2m},(k+1)/2^{2m}],\\
J(x_0)&:=T(I(x_0)),\\
H(x_0)&:=\{(x,T(x)): x\in I(x_0)\}.
\end{align*}
By Lemma \ref{lem:similar-copies}, $H(x_0)$ is a similar copy of $\mathcal{G}_T$; we call it a {\em hump}. If $D_j(x_0)\geq 0$ for every $j\leq 2m$, we call $H(x_0)$ a {\em leading hump}.
}
\end{definition}

We recall here that the range of $T$ is $[0,\frac23]$. Thus, if $x_0=k/2^{2m}$ is balanced of order $m$, then $\diam(J(x_0))=\frac23{(\frac14)}^m$. 

Next, define a subset $X^*$ of $[0,1]$ by
\begin{equation}
X^*:=[0,1]\backslash\bigcup_{x_0\in\BB_1}I(x_0).
\label{eq:Xstar}
\end{equation}
The importance of $X^*$ lies in the following result; see \cite[Proposition 3.2]{Allaart}.

\begin{proposition} \label{lem:bijection}
We have $T(X^*)=[0,\frac12]$, and $T$ is strictly increasing on $X^*\cap[0,\frac12)$.
\end{proposition}

For $x_0\in\BB_n$, define a subset $X^*(x_0)$ of $I(x_0)$ similarly by
\begin{equation*}
X^*(x_0):=I(x_0)\backslash \bigcup_{x_1\in\BB_{n+1}} I(x_1).
\end{equation*}
We call the graph of $T$ restricted to $X^*(x_0)$ a {\em truncated hump}, and denote it by $H^t(x_0)$.
Let $J^t(x_0)$ be the orthogonal projection of $H^t(x_0)$ onto the $y$-axis, so $J^t(x_0)=T(X^*(x_0))$. Proposition \ref{lem:bijection} and Lemma \ref{lem:similar-copies} imply that $J^t(x_0)$ is an interval; in particular, if $J(x_0)=[a,a+\frac23{(\frac14)}^m]$, then $J^t(x_0)=[a,a+\frac12{(\frac14)}^m]$.

\section{Correction and a new result}

Section 4 of \cite{Allaart} concerns in particular the sets
\begin{gather*}
S_\infty^{co}:=\{y: L(y)\ \mbox{is countably infinite}\},\\
S_\infty^{uc}:=\{y: L(y)\ \mbox{is uncountably infinite}\}.
\end{gather*}
Lemma 4.3 gives an explicit representation of $S_\infty^{uc}$. However, there is a logical flaw in lines 2-4 of its proof, and the author does not know at this point whether the lemma is true. The main impact on the results of \cite{Allaart} is that Theorem 4.6, which states that $S_\infty^{uc}$ is the union of a dense $G_\delta$ set and a countable set, can no longer be justified. The same is true for the conclusion, in Theorem 4.6 and Corollary 4.7, that the sets $S_\infty^{co}$ and $S_\infty^{uc}$ are ${\bf\Delta}_3^0$ in the Borel hierarchy. The author does not know whether these statements are true, or whether $S_\infty^{co}$ and $S_\infty^{uc}$ are even Borel sets. However, the explicit description of $S_\infty^{uc}$ is not needed to prove the following: 

\begin{theorem}
The set $S_\infty^{uc}$ is residual in the range of $T$.
\end{theorem}

\begin{proof}
Since $T$ is nowhere differentiable, it is monotone on no interval. Thus, the result follows from Theorem 1 of Garg \cite{Garg}, which states that for {\em any} continuous function $f$ which is monotone on no interval, the level set at level $y$ is perfect (and hence uncountable) for a set of $y$-values residual in the range of $f$.
\end{proof}

Since Lemma 4.3 was used to prove Theorem 5.1, we provide here an alternative argument for this result. In what follows, let $l_y$ denote the horizontal line at level $y$. That is,
\begin{equation*}
l_y:=\{(x,y): x\in\RR\}.
\end{equation*}

\begin{theorem} 
Every uncountable level set contains an uncountable local level set.
\end{theorem}

\begin{proof}
In the proof of Theorem 3.11, a bijection was constructed between the collection of finite local level sets in $L(y)$ and the collection of truncated leading humps which intersect $l_y$. Since there are only countably many truncated humps, it follows that each level set contains only countably many finite local level sets, so if a level set is uncountable then it must contain an uncountable local level set.
\end{proof}

We end this note with a strengthening of Theorem 5.2 of \cite{Allaart} about the propensity of level sets containing infinitely many local level sets. Let
\begin{gather*}
S_\infty^{loc}:=\{y: L(y)\ \mbox{\rm contains infinitely many distinct local level sets}\},\\
S_\infty^{loc,uc}:=\{y: L(y)\ \mbox{\rm contains {\em uncountably many} distinct local level sets}\}.
\end{gather*}
Lagarias and Maddock \cite{LagMad1} asked whether $S_\infty^{loc}$ is countable. 
Theorem 5.2 of \cite{Allaart} gives a negative answer to this question by showing that $S_\infty^{loc}$ is residual in $[0,\frac23]$, and $S_\infty^{loc,uc}$ is dense in $[0,\frac23]$ and intersects each subinterval of $[0,\frac23]$ in a continuum. Using Garg's theorem, we can now prove an even stronger result:

\begin{theorem} \label{thm:uncountable-local}
The set $S_\infty^{loc,uc}$ is residual in the range of $T$.
\end{theorem}

The proof uses the following lemma, in which $M$ denotes the set of right endpoints of the intervals $J(x_0)$, $x_0\in\BB$.

\begin{lemma} \label{lem:truncated-hump}
\begin{enumerate}[(i)]
\item If $0\leq y<\frac23$, then the line $l_y$ intersects a truncated leading hump.
\item If $y\not\in M$ and $l_y$ intersects a leading hump $H$, then $l_y$ intersects a truncated leading hump contained in $H$.
\end{enumerate}
\end{lemma}

\begin{proof}
Let $x_0=y_0=0$, and $x_n=\sum_{i=1}^n {(\frac14)}^i$, $y_n=\sum_{i=1}^n {(\frac12)}^i$ for $n\in\NN$. Then $x_n\in\BB$ and $D_j(x_n)\geq 0$ for each $j$, so $H(x_n)$ is a leading hump. Furthermore, a direct calculation or a look at the graph of $T$ reveals that $J^t(x_n)=[y_n,y_{n+1}]$. Since $\bigcup_{n=0}^\infty [y_n,y_{n+1}]=[0,\frac23)$, it follows that if $0\leq y<\frac23$, then $y\in J^t(x_n)$ for some $n$, and by Proposition \ref{lem:bijection} and Lemma \ref{lem:similar-copies}, $l_y$ intersects the truncated hump $H^t(x_n)$. This proves (i). Statement (ii) follows from (i) and the following, easily verified fact: If $H_1$ and $H_2$ are leading humps and $S_1,S_2:\RR^2\to\RR^2$ are orientation-preserving similitudes such that $S_i(\mathcal{G}_T)=H_i$ for $i=1,2$, then $H_3:=S_1\big(S_2(\mathcal{G}_T)\big)$ is also a leading hump. 
\end{proof}

\begin{proof}[Proof of Theorem \ref{thm:uncountable-local}]
Let $S_\infty^p:=\{y\in[0,\frac23]: L(y)\ \mbox{is a perfect set}\}$, let $\mathcal{D}$ be the set of dyadic rationals in $[0,1]$, and let $M$ be as in Lemma \ref{lem:truncated-hump}. Let 
\begin{equation*}
E:=S_\infty^p\backslash\big(M\cup T(\mathcal{D})\big). 
\end{equation*}
By Garg's theorem \cite[Theorem 1]{Garg}, $S_\infty^p$ is residual in $[0,\frac23]$, and since $M$ and $\mathcal{D}$ are countable, $E$ is residual in $[0,\frac23]$ as well. Let
\begin{equation*}
X_0:=\{x\in[0,1): D_n(x)\geq 0\ \mbox{for every}\ n\}.
\end{equation*}
We will show that
\begin{equation}
\mbox{for each}\ y\in E, L(y)\cap X_0\ \mbox{is a perfect set}.
\label{eq:perfect-sets}
\end{equation}
This will clearly yield the theorem, since it implies that $L(y)\cap X_0$ is uncountable for all $y\in E$, and different members of $X_0$ represent different local level sets.

Let $y\in E$, and let $x\in L(y)\cap X_0$. We must show that $x$ is a limit point of $L(y)\cap X_0$. To this aim, consider two cases:

\bigskip
{\em Case 1:} $L_x^{loc}$ is finite. Since $x\in L(y)$ and $y\in S_\infty^p$, there is a sequence $\{x_n\}$ in $L(y)$ such that $x_n\neq x$ for each $n$, and $x_n\to x$. Replacing this sequence with a subsequence if necessary, we may assume that $x,x_1,x_2,\dots$ all represent different local level sets, because $L_x^{loc}$ is finite and local level sets are closed. Since $y\in E$, $x\not\in \mathcal{D}$ and so for each $m\in\NN$ there is an integer $N_m$ such that the binary expansion of $x_n$ agrees with that of $x$ up to and including the $m$th digit for all $n\geq N_m$. For given $n$, let $x_n'$ be the point such that $D_j(x_n')=|D_j(x_n)|$ for every $j$; then $x_n'\in L(y)\cap X_0$ by \eqref{eq:equivalence}, $x_n'$ belongs to the same local level set as $x_n$, and if $n\geq N_m$ then $|x_n'-x_n|\leq 2^{-m}$ since $x\in X_0$. Hence, $x_n'\to x$, and clearly, $x_n'\neq x$ for each $n$.

\bigskip
{\em Case 2:} $L_x^{loc}$ is a Cantor set. Then $D_j(x)=0$ for infinitely many $j$, and for each $n\in\NN$ there is a balanced dyadic rational $x_n\in\BB_n$ in such a way that $I(x_{n+1})\subset I(x_n)$ for all $n$, and $\bigcap_{n=1}^\infty I(x_n)=\{x\}$. Since $x\in X_0$, $H(x_n)$ is a leading hump for each $n$. Since $y\not\in M$, Lemma \ref{lem:truncated-hump} implies that for each $n$, the line $l_y$ intersects a truncated leading hump $H_n^t$ contained in $H(x_n)$. By Proposition \ref{lem:bijection} and Lemma \ref{lem:similar-copies}, it does so in exactly two points; let $z_n$ be the leftmost of these. 
Then $z_n\in L(y)$, and $|z_n-x|\leq \diam(I(x_{n}))\to 0$, so $z_n\to x$. Furthermore, since $(z_n,y)$ is the leftmost point of $l_y\cap H_n^t$ and $H_n^t$ is a truncated leading hump, we have $z_n\in X_0$ and $D_j(z_n)=0$ for only finitely many $j$, so $z_n\neq x$.

\bigskip
We have shown in both cases that $x$ is a limit point of $L(y)\cap X_0$. Thus we have proved \eqref{eq:perfect-sets}, and the theorem.
\end{proof}

Theorem \ref{thm:uncountable-local} is quite remarkable, since the {\em average} number of local level sets contained in a level set (with respect to Lebesgue measure on the range of $T$) is $3/2$; see \cite{LagMad1} or \cite[Theorem 3.11]{Allaart}. This implies that at least `half' of all level sets consist of just one local level set. Yet the `typical' level set contains uncountably many local level sets.

\section*{Acknowledgments}
The author is grateful to the referee of \cite{Allaart2} for pointing out Garg's paper \cite{Garg}.

\end{document}